\documentclass[11pt,letterpaper]{amsart}
\usepackage[T1]{fontenc}
\usepackage[utf8]{inputenc}
\usepackage[english]{babel}
\usepackage{lmodern}
\usepackage{microtype}
\usepackage[margin=2.65cm]{geometry}
\usepackage{amssymb,mathtools}
\usepackage{booktabs,array}
\usepackage{enumitem}
\usepackage{needspace}
\usepackage{xurl}
\usepackage[hidelinks]{hyperref}
\hypersetup{pdftitle={Homotopical Nilpotency and Homology Nilpotency: A Dimension-Connectivity Bound},pdfauthor={Paul-Eugène Parent}}
\setlist{itemsep=0.2em,topsep=0.4em,leftmargin=2em}
\newcommand{\Q}{\mathbb Q}
\newcommand{\mm}{\mathfrak m}
\newcommand{\nilh}{\operatorname{nil}_{h}}
\newcommand{\Hnil}{\operatorname{Hnil}}
\newcommand{\cupL}{\operatorname{cup}}
\newcommand{\cdim}{\operatorname{cd}_{\Q}}

\newcommand{\Span}{\operatorname{span}_{\Q}}
\newcommand{\Sym}{\operatorname{Sym}}
\newcommand{\ideal}{\mathrel{\triangleleft}}
\newcommand{\qis}{\xrightarrow{\simeq}}
\newcommand{\sqis}{\xrightarrow[\simeq]{\twoheadrightarrow}}

\newtheorem{theorem}{Theorem}[section]
\newtheorem{proposition}[theorem]{Proposition}
\newtheorem{lemma}[theorem]{Lemma}
\newtheorem{corollary}[theorem]{Corollary}
\theoremstyle{definition}
\newtheorem{definition}[theorem]{Definition}
\theoremstyle{remark}
\newtheorem{remark}[theorem]{Remark}
\newtheorem{example}[theorem]{Example}
\numberwithin{equation}{section}
\title[Homotopical and homology nilpotency]{Homotopical Nilpotency and Homology Nilpotency: A Dimension--Connectivity Bound}
\author[Paul-Eugène Parent]{Paul-Eugène Parent\textsuperscript{*}}
\address{Department of Mathematics and Statistics, Faculty of Science,
University of Ottawa, STEM Complex, Room 361, 150 Louis-Pasteur Private,
Ottawa, Ontario K1N 6N5, Canada}
\email{pparent@uottawa.ca}
\thanks{\textsuperscript{*}This research was supported by the Natural Sciences and Engineering Research Council of Canada (NSERC), grant no.~149521.}
\date{September 6, 2026}

\begin{document}
\begin{abstract}
Let $X$ be a $(q-1)$-connected rational space of finite type, where $q\ge2$, with homotopical nilpotency $\nilh(X)=n\ge1$. If $H^{>N}(X;\Q)=0$ and $N\le q(n+3)-3$, we prove that $\Hnil(X)=n$. More precisely, the sufficient bound is $N\le a_n(M_X)+2q-3$, where $a_n(M_X)$ is the first nonzero internal cohomology degree of $(M_X^+)^{n+1}$ in the minimal model. A family of examples proves that the constant $-3$ in this refined bound is optimal; optimality of the uniform bound is not asserted. The proof uses two stages of adjoining primitives to an ideal of the fixed minimal model.
\end{abstract}

\maketitle

\section{Introduction}
\label{sec:intro}

This paper constructs short quotients of minimal Sullivan algebras from the existence of short models. The construction enlarges a differential ideal inside the fixed minimal algebra, and a dimension--connectivity estimate ensures that the resulting ideal is acyclic. It yields an equality criterion for homotopical and homology nilpotency, together with a refined bound whose constant is optimal.

Let $M=(\Lambda V,d)$ be a simply connected minimal Sullivan algebra of finite type. An augmented commutative differential graded algebra, or CDGA, is called \emph{$n$-short} if the $(n+1)$st power of its augmentation ideal is zero. The \emph{homotopical nilpotency} $\nilh(M)$ is the least $n$ for which the quasi-isomorphism type of $M$ contains an $n$-short CDGA. For a simply connected rational space $X$ of finite type with minimal model $M$, Cornea identifies $\nilh(M)$ with the rational cone-length of $X$ \cite{CorneaRational}.

The invariant $\Hnil(M)$ requires an $n$-short quotient of $M$. It is the least $n$ for which an acyclic differential ideal $J$ of $M$ contains $(M^+)^{n+1}$, or equivalently, for which there is a surjective quasi-isomorphism
\[
 M\sqis B,\qquad (B^+)^{n+1}=0.
\]
This is Carrasquel-Vera's homology nilpotency of an ideal \cite[Section~4]{CarrasquelVera2015}, applied to the augmentation ideal of the minimal model. Every short quotient is a short model, so $\nilh(M)\le\Hnil(M)$. Equality asks whether the shortest multiplication bound available in the rational homotopy type can be realized by a quotient of its minimal model.

The author's monograph \cite[Chapters~4, 11--12]{Parent2026} develops acyclic-ideal methods and establishes the separation $\nilh=3<4=\Hnil$. It also proves equality when $\nilh\le2$ \cite[Corollary~6.7]{Parent2026}. The present contribution concerns the control of products of primitives: a relative two-stage construction uses a comparison quasi-isomorphism to extend the range in which the required acyclic ideal can be built. This gives a uniform equality criterion, of particular interest from length three onward, and a refinement whose constant is optimal.

Connectivity is expressed by $V^{<q}=0$, and the dimension condition is $H^{>N}(M)=0$. We assume neither formality nor coformality nor Poincaré duality. The use of dimension and connectivity to improve categorical coverings in the work of Clapp and Puppe \cite{CP87} inspires the question studied here.

\Needspace{18\baselineskip}\subsection{Main result}

\begin{theorem}[Uniform dimension--connectivity bound]\label{thm:main}
Let $M=(\Lambda V,d)$ be a minimal Sullivan algebra over $\Q$, of finite type. Suppose that
\[
 V^{<q}=0,\qquad q\ge2,\qquad N\ge0,\qquad H^{>N}(M)=0.
\]
Assume that $\nilh(M)=n$ for an integer $n\ge1$. Then
\begin{equation}\label{eq:main}
 \boxed{\quad N\le q(n+3)-3
 \quad\Longrightarrow\quad\Hnil(M)=n.\quad}
\end{equation}
More precisely, there is an acyclic differential ideal $J\ideal M$ containing $(M^+)^{n+1}$ and $M^{>N}$. The quotient $M\sqis M/J$ is therefore $n$-short and finite-dimensional as a vector space.
\end{theorem}

The elementary truncation bound in \cite[Proposition~4.4]{Parent2026} already gives $\Hnil(M)\le n$ when $N<q(n+1)$: all words of length $n+1$ then lie above the top cohomological degree. Theorem~\ref{thm:main} uses the short-model hypothesis to extend this bound by $2q-2$ degrees. In this additional range, the long-word ideal can have nonzero internal cohomology, and its elimination requires control of the products forced by adjoining primitives.

For $n=3$, the bound is $N\le6q-3$. In particular, it gives equality in the simply connected case through cohomological dimension $9$, and in the $2$-connected case through cohomological dimension $15$.

The theorem follows from a more precise statement. Set
\[
 a_n(M)=\inf\{j\mid H^j((M^+)^{n+1})\ne0\}.
\]
This number measures the first nonzero \emph{internal} cohomology of the long-word ideal, not the image of that cohomology in $H(M)$. We shall prove that
\begin{equation}\label{eq:refined-intro}
 N\le a_n(M)+2q-3
 \quad\Longrightarrow\quad\Hnil(M)=n.
\end{equation}
The uniform bound follows from $a_n(M)\ge q(n+1)$. Theorem~\ref{thm:refined} gives the corresponding upper bound at an arbitrary level $n$.

Two families clarify the scope of the result. The first shows that cohomological injectivity of the projection onto short words, measured by the Toomer invariant, cannot replace the hypothesis on $\nilh$ in the new range. The second regrades and truncates the separating pattern of \cite[Chapters~11--12]{Parent2026}. Its new role is to test the exact endpoint of the refined bound: the proof controls every possible correction to the critical primitive. For every odd $q\ge3$, it gives a minimal model $M_q$ with
\[
 \cdim M_q=7q-3=a_3(M_q)+2q-2,
 \qquad \nilh(M_q)=3<4=\Hnil(M_q).
\]
Thus equality fails at the first integer beyond the refined bound, proving that its constant $-3$ cannot be replaced universally by $-2$. These models do not establish the optimality of the uniform bound~\eqref{eq:main}.

\subsection{The construction}

The obstruction is internal to the ideal. Making its cocycles exact in $M$ does not suffice: adjoining their primitives to an ideal also forces their products into that ideal, and these products can create new cohomology. Theorem~\ref{thm:two} controls this process for a differential ideal annihilated by a quasi-isomorphism, under explicit degree hypotheses. The application to short models takes $I=(M^+)^{n+1}$.

Choose a quasi-isomorphism $f:M\to B$ to an $n$-short witness; surjectivity of $f$ is not required. Since $f(I)=0$, the internal classes of $I$ have primitives in $M$. Surjectivity of $H(f)$ allows these primitives to be chosen so that their images are boundaries in $B$. After they are adjoined to $I$, the coefficients of their products through degree $N$ have degree at most $2q-2$ and are closed by minimality. The resulting cycles therefore still map to boundaries in $B$ and remain exact in $M$.

A second family of primitives kills the remaining classes through degree $N$. Its products with positive-degree coefficients start in degree $a_n(M)+2q-2$, above the allowed top cohomological degree. An acyclic completion then produces the desired ideal in all degrees. Every primitive belongs to the original algebra, whose multiplication, differential and minimality are preserved throughout.

Sections~\ref{sec:completion} and~\ref{sec:two} establish the ideal constructions, and Section~\ref{sec:main-proof} applies them to the nilpotency invariants. The two families of examples follow. We use the standard existence and comparison results for Sullivan models from \cite[Chapters~12 and~14]{FHT}; the extension from a prescribed minimal prefix is proved in the appendix.

\section{Conventions and elementary comparisons}
\label{sec:conventions}

All CDGAs are defined over $\Q$, augmented, cohomologically graded in nonnegative degrees, and equipped with a differential of degree $+1$. Unless otherwise stated, they are connected, meaning that $A^0=\Q$. Commutativity is graded: $uv=(-1)^{|u||v|}vu$. We write
\[
 \mm_A=A^+=\ker(A\to\Q),\qquad
 Z^j(A)=\ker(d:A^j\to A^{j+1}),\qquad
 B^j(A)=d(A^{j-1}).
\]
An ideal is always graded; it is \emph{differential} if it is stable under $d$, and \emph{acyclic} if its cohomology as a nonunital complex vanishes.

All finite degree bounds and nilpotency levels are integers. We use $\inf\varnothing=+\infty$ for the initial-degree invariants below.

For a minimal algebra $M=(\Lambda V,d)$, the condition $V^{<q}=0$ implies
\begin{equation}\label{eq:degree-estimate}
 M^j=0\ (1\le j<q),\qquad
 (\mm_M^r)^j=0\ (j<qr),\qquad
 d(\mm_M^r)\subseteq\mm_M^{r+1}.
\end{equation}
The last inclusion follows from the decomposability of $dV$ and the Leibniz rule. The word \emph{length} always refers to augmentation length, not cohomological degree.

\begin{definition}\label{def:invariants}
For a simply connected minimal algebra $M$, set
\begin{align*}
 \nilh(M)&=\inf\{n\ge0\mid M\simeq B,\ (B^+)^{n+1}=0\},\\
 \Hnil(M)&=\inf\{n\ge0\mid \exists J\ideal M,\ H(J)=0,
                   \ (M^+)^{n+1}\subseteq J\}.
\end{align*}
The symbol $\simeq$ denotes a zigzag of augmented quasi-isomorphisms. The infimum of an empty set is $+\infty$.
\end{definition}
The second invariant is Carrasquel-Vera's homology nilpotency of the ideal $M^+$ in the ambient CDGA $M$ \cite[Section~4]{CarrasquelVera2015}.

\begin{lemma}[Choosing a direct witness]\label{lem:direct}
If $\nilh(M)\le n$, there is an augmented quasi-isomorphism $f:M\to B$ with $(B^+)^{n+1}=0$. No surjectivity of $f$ on cochains is asserted.
\end{lemma}
\begin{proof}
Choose an $n$-short witness $B$. Sullivan models are invariant under zigzags of quasi-isomorphisms \cite[Proposition~12.9, pp.~153--154, and Section~12(c), p.~156]{FHT}; hence the given minimal algebra $M$ admits a quasi-isomorphism to $B$. It is augmented, since both algebras are connected and the morphism preserves degrees. No surjectivity on cochains is involved.
\end{proof}

\begin{lemma}[Quotient characterization]\label{lem:quotient}
The condition $\Hnil(M)\le n$ is equivalent to the existence of an augmented surjective quasi-isomorphism $M\sqis C$ such that $(C^+)^{n+1}=0$. In particular,
\begin{equation}\label{eq:nil-comparison}
 \nilh(M)\le\Hnil(M).
\end{equation}
\end{lemma}
\begin{proof}
If $J$ is an ideal as in the definition, the short exact sequence of complexes
\[
 0\longrightarrow J\longrightarrow M\longrightarrow M/J\longrightarrow0
\]
shows that $M\to M/J$ is a quasi-isomorphism. The $(n+1)$st power of the augmentation ideal of the quotient is zero. Conversely, the kernel of a surjective quasi-isomorphism is acyclic by the same exact sequence. If its target is $n$-short, it contains $(M^+)^{n+1}$. Finally, a short quotient is an admissible witness for $\nilh$.
\end{proof}

An isomorphism of minimal CDGAs transports the augmentation ideal, its powers, and the acyclic ideals. The invariants defined above therefore do not depend on the choice of coordinates in the minimal model. For a simply connected rational space $X$ with minimal model $M_X$, we write $\nilh(X)=\nilh(M_X)$ and $\Hnil(X)=\Hnil(M_X)$.

\begin{definition}\label{def:internal}
The Toomer invariant is defined by
\[
 e_0(M)=\inf\{n\ge0\mid H(M\to M/\mm_M^{n+1})
                      \text{ is injective}\}.
\]
For $n\ge1$, we also set
\begin{align}
 \nu_{n+1}(M)&=\inf\{j\mid(\mm_M^{n+1})^j\ne0\},\label{eq:nu}\\
 a_n(M)&=\inf\{j\mid H^j(\mm_M^{n+1})\ne0\}.\label{eq:a}
\end{align}
The cohomological dimension is $\cdim M=\sup\{j\mid H^j(M)\ne0\}$.
\end{definition}

We write $\cupL H(M)$ for the cup-length: the supremum of the lengths of nonzero products of positive-degree cohomology classes, with value $0$ if $H^+(M)=0$ and $+\infty$ if these lengths are unbounded.

\begin{proposition}\label{prop:toomer}
The condition $e_0(M)\le n$ is equivalent to the vanishing of the map $H(\mm_M^{n+1})\to H(M)$. Moreover,
\begin{equation}\label{eq:toomer-chain}
 \cupL H(M)\le e_0(M)\le\nilh(M)\le\Hnil(M).
\end{equation}
Under the assumption $V^{<q}=0$, we have
\begin{equation}\label{eq:a-bound}
 q(n+1)\le\nu_{n+1}(M)\le a_n(M).
\end{equation}
\end{proposition}
\begin{proof}
Set $I=\mm_M^{n+1}$. If the projection is injective on cohomology, every cocycle of $I$ is exact in $M$. Conversely, suppose this exactness holds. If a cocycle $z$ of $M$ becomes a boundary in $M/I$, there is $v\in M$ such that $z-dv\in I$. This element is closed and hence exact in $M$; therefore $z$ is exact.

A quasi-isomorphism to an $n$-short target annihilates $I$. Lemma~\ref{lem:direct} therefore gives $e_0\le\nilh$. If a product of $r$ positive-degree classes is nonzero, the product of closed representatives belongs to $\mm_M^r$ and represents an essential class. Its projection to $M/\mm_M^r$ is zero, so $e_0(M)\ge r$. The last inequality in~\eqref{eq:toomer-chain} is~\eqref{eq:nil-comparison}. Finally,~\eqref{eq:degree-estimate} gives the first inequality in~\eqref{eq:a-bound}; the second is immediate.
\end{proof}

\begin{remark}\label{rem:internal}
The hypothesis $\nilh(M)\le n$ kills the image of $H(I)$ in $H(M)$, but it need not kill $H(I)$ itself. A cocycle may be a boundary in $M$ without having a primitive in $I$. The number $a_n$ measures exactly the first degree in which this internal phenomenon can occur. Likewise, an essential cycle of length $r$ bounds $e_0$ below by $r$ even when its factors are not individually closed: this observation is more general than a lower bound obtained from cup products.
\end{remark}

\section{Internal primitives and acyclic completion}
\label{sec:completion}

The constructions in this section refine the acyclic-ideal methods of \cite[Chapter~4]{Parent2026}. The completion lemma is a relative form of the standard finite-dimensional truncation construction; compare \cite[Section~12(a), Example~6, pp.~146--147]{FHT}. We give the proofs to make the degree estimates and the later two-stage argument self-contained.

\begin{lemma}[Completion above the top cohomological degree]\label{lem:completion}
Let $A$ be a connected CDGA of finite type, with $H^{>N}(A)=0$, $N\ge1$. If a differential ideal $J_0\ideal A$ satisfies $H^j(J_0)=0$ for $j\le N$, there is an acyclic differential ideal $J$ such that
\[
 J_0+A^{>N}\subseteq J.
\]
\end{lemma}
\begin{proof}
Set $K=J_0+A^{>N}$. Since degrees are nonnegative, this is a differential ideal. For $j\le N$, its cochains and boundaries are those of $J_0$. In degree $N$, enlarging the target component of the differential does not change its kernel on $J_0^N$: being closed still means mapping to zero in $A^{N+1}$. Consequently, $H^{\le N}(K)=0$.

Choose cocycles $z_1,\ldots,z_t\in K^{N+1}$ representing a basis of $H^{N+1}(K)$. They are exact in $A$, since $H^{N+1}(A)=0$. Choose $u_i\in A^N$ with $du_i=z_i$ and set
\[
 J=K+\Span\{u_1,\ldots,u_t\}.
\]
This subspace is an ideal: every product $u_iA^+$ lies in $A^{>N}$. It is stable under $d$. The classes of the $u_i$ modulo $K$ are independent. Indeed, a relation $\sum_i\lambda_i u_i\in K$, after applying $d$, would give $\sum_i\lambda_i[z_i]=0$ in $H^{N+1}(K)$, hence $\lambda_i=0$ for every $i$.

Thus $J/K$ is concentrated in degree $N$ and has zero differential. The connecting map of
\[
 0\longrightarrow K\longrightarrow J\longrightarrow J/K\longrightarrow0
\]
sends the class of $u_i$ to that of $z_i$ and is an isomorphism in these degrees. The long exact sequence gives $H^{\le N+1}(J)=0$. If $j>N+1$, every primitive in $A^{j-1}$ of a degree-$j$ cocycle already belongs to $A^{>N}\subseteq J$. Since $H^j(A)=0$, we also have $H^j(J)=0$. The case $t=0$ is included in this argument.
\end{proof}

\begin{corollary}[Degree truncation]\label{cor:degree}
Let $M$ be minimal of finite type, and let $n,N\ge0$. Suppose that $H^{>N}(M)=0$. If $(\mm_M^{n+1})^{\le N}=0$, then $\Hnil(M)\le n$. The short quotient can be chosen to vanish above degree $N$.
\end{corollary}
\begin{proof}
For $N\ge1$, apply the lemma to $J_0=\mm_M^{n+1}$. If $N=0$, the ideal $M^+$ is acyclic and the quotient $M\to\Q$ has the required properties.
\end{proof}

The next operation adjoins \emph{elements of the fixed algebra} to an ideal. It adjoins no generators to the algebra, so there is no need to reconstruct and preserve minimality at each step.

\begin{lemma}[First possible degree of new classes]\label{lem:adjoin}
Let $A$ be a connected CDGA with $A^j=0$ for $1\le j<q$, $q\ge2$, and let $I\ideal A$ be a differential ideal such that $H^{<a}(I)=0$, where $a\ge2$. For each $a\le s\le N$, choose cocycles $z_{s,i}$ representing a basis of $H^s(I)$, and elements $u_{s,i}\in A^{s-1}$ with $du_{s,i}=z_{s,i}$. Set
\[
 J_1=I+\sum_{s,i}A u_{s,i},\qquad b=a+q-1.
\]
Then $J_1$ is a differential ideal and
\[
 H^j(J_1)=0\qquad(j\le N,\ j<b).
\]
\end{lemma}
\begin{proof}
Stability under the differential follows from $du_{s,i}\in I$ and the Leibniz rule. Every adjoined primitive has degree at least $a-1$. A product with a positive-degree coefficient therefore starts in degree $b$.

Fix $j\le N$ with $j<b$. Every element $z\in J_1^j$ can be written as
\[
 z=v+\sum_i\lambda_i u_{j+1,i},\qquad v\in I^j,
\]
where the sum is empty if no generator $u_{j+1,i}$ was chosen. If $dz=0$, we obtain
\[
 dv+\sum_i\lambda_i z_{j+1,i}=0.
\]
The independence of the classes $[z_{j+1,i}]$ in $H^{j+1}(I)$ forces all the $\lambda_i$ to vanish. Thus $z=v$ is a cocycle of $I^j$. If $j<a$, it is already exact in $I$. If $a\le j\le N$, there are $w\in I^{j-1}$ and scalars $\mu_i$ such that
\[
 v=dw+\sum_i\mu_i z_{j,i}
   =d\left(w+\sum_i\mu_i u_{j,i}\right).
\]
It is therefore exact in $J_1$. This argument also covers $j=N$, without assuming that a basis in degree $N+1$ has been chosen.
\end{proof}

\begin{proposition}[Elimination before the first products]\label{prop:one}
Under the degree assumptions of Lemma~\ref{lem:adjoin}, suppose that $A$ is of finite type, $N\ge1$, $H^{>N}(A)=0$, $H^{<a}(I)=0$, and $H^s(I)\to H^s(A)$ is zero for $s\le N$. If
\[
 N\le a+q-2,
\]
then $I$ is contained in an acyclic differential ideal of $A$ containing $A^{>N}$.
\end{proposition}
\begin{proof}
If $a>N$, apply Lemma~\ref{lem:completion} directly. Otherwise, the cohomological vanishing allows the primitives in Lemma~\ref{lem:adjoin} to be chosen. Since $N<a+q-1$, that lemma gives $H^{\le N}(J_1)=0$. Lemma~\ref{lem:completion} completes the proof.
\end{proof}

\begin{corollary}[The first range from Toomer's invariant]\label{cor:toomer}
Let $M$ be minimal of finite type, with $V^{<q}=0$, $q\ge2$, $N\ge0$, and $H^{>N}(M)=0$. For $n\ge1$, we have
\[
 e_0(M)\le n,\qquad N\le a_n(M)+q-2
 \quad\Longrightarrow\quad\Hnil(M)\le n.
\]
In particular, the condition $N\le q(n+2)-2$ is sufficient.
\end{corollary}
\begin{proof}
If $H^+(M)=0$, the ideal $M^+$ suffices. Otherwise, $N\ge1$. If $a_n=+\infty$, the ideal $\mm_M^{n+1}$ is acyclic. In the remaining case, Propositions~\ref{prop:toomer} and~\ref{prop:one} apply to this ideal, with $a=a_n$. Inequality~\eqref{eq:a-bound} gives the uniform version.
\end{proof}

\Needspace{9\baselineskip}\section{The relative two-stage theorem}
\label{sec:two}

\begin{lemma}[Low-degree coefficients]\label{lem:closed}
For a minimal algebra $M=(\Lambda V,d)$ with $V^{<q}=0$, we have
\[
 d(M^j)=0\qquad(0\le j\le2q-2).
\]
\end{lemma}
\begin{proof}
The differential of a positive-degree element is decomposable. A nonzero decomposable element has degree at least $2q$, whereas $j+1<2q$ in the stated range. The differential of the unit is also zero.
\end{proof}

\begin{lemma}[Normalization of a primitive]\label{lem:normalize}
Let $f:A\qis B$ be a quasi-isomorphism. If $du=z$ and $f(z)=0$, there is a cocycle $c\in A^{|u|}$ such that
\[
 d(u-c)=z,\qquad f(u-c)=d\beta
\]
for some element $\beta\in B^{|u|-1}$.
\end{lemma}
\begin{proof}
The element $f(u)$ is closed. The surjectivity of $H(f)$ provides a cocycle $c$ whose image represents the same class. The difference $f(u)-f(c)$ is therefore a boundary. Since $dc=0$, the differential of the primitive is unchanged.
\end{proof}

\begin{remark}
The lemma does not make $f(u-c)$ zero. Nor does it assert that an arbitrary cochain of $B$ can be lifted. Only surjectivity on cohomology is used, and this is available before any surjective strictification.
\end{remark}

\Needspace{14\baselineskip}\begin{theorem}[Relative two-stage elimination]\label{thm:two}
Let $A$ be a connected CDGA of finite type satisfying
\[
 A^j=0\ (1\le j<q),\qquad q\ge2,
 \qquad d(A^{\le2q-2})=0,
 \qquad H^{>N}(A)=0,
\]
where $N\ge1$. Let $f:A\qis B$ be a quasi-isomorphism, and let $I\ideal A$ be a differential ideal such that $f(I)=0$. Suppose that
\[
 H^{<a}(I)=0,\qquad a\ge2.
\]
If $N\le a+2q-3$, there is an acyclic differential ideal $J\ideal A$ with
\[
 I+A^{>N}\subseteq J.
\]
\end{theorem}
\begin{proof}
If $a>N$, the ideal $I$ is acyclic through degree $N$, and Lemma~\ref{lem:completion} gives the conclusion. Assume that $a\le N$.

\emph{First stage.} For each $a\le s\le N$, choose cocycles $z_{s,i}\in I^s$ representing a basis of $H^s(I)$. They are exact in $A$: their images under $f$ are zero and $H(f)$ is injective. Choose primitives and then apply Lemma~\ref{lem:normalize} to obtain
\begin{equation}\label{eq:normalized}
 du_{s,i}=z_{s,i},\qquad f(u_{s,i})=d\beta_{s,i}.
\end{equation}
Set $J_1=I+\sum_{s,i}A u_{s,i}$. This is a differential ideal.

\emph{Exactness of the first products in the relevant range.} If a homogeneous term $c u_{s,i}$ has degree at most $N$, then
\begin{equation}\label{eq:coefficient-degree}
 |c|\le N-(s-1)\le N-a+1\le2q-2.
\end{equation}
By hypothesis, $dc=0$. The Leibniz rule and~\eqref{eq:normalized} give
\[
 f(cu_{s,i})=f(c)d\beta_{s,i}
           =(-1)^{|c|}d\bigl(f(c)\beta_{s,i}\bigr).
\]
Since $f(I)=0$, every element of $J_1^j$, for $j\le N$, maps to a boundary in $B$. In particular,
\begin{equation}\label{eq:first-essential}
 H^j(J_1)\longrightarrow H^j(A)\quad\text{is zero for }j\le N.
\end{equation}
We have not used the assertion, false in general, that the space of boundaries of $B$ is an ideal. Here the coefficients multiplying boundaries are explicitly closed.

\emph{The degree where new cohomology can first appear.} Write $b=a+q-1$. Lemma~\ref{lem:adjoin} gives
\begin{equation}\label{eq:first-vanishing}
 H^j(J_1)=0\qquad(j\le N,\ j<b).
\end{equation}
Thus the first classes that may survive occur only in degree $b$.

\emph{Second stage.} For $b\le s\le N$, choose cocycles $w_{s,i}$ representing a basis of $H^s(J_1)$. By~\eqref{eq:first-essential}, they are exact in $A$. Choose $v_{s,i}\in A^{s-1}$ with $dv_{s,i}=w_{s,i}$ and set
\[
 J_2=J_1+\sum_{s,i}A v_{s,i}.
\]
The new products with positive-degree coefficients have degree at least
\begin{equation}\label{eq:second-product}
 (b-1)+q=a+2q-2>N.
\end{equation}
If $b\le N$, the groups $H^{<b}(J_1)$ vanish by~\eqref{eq:first-vanishing}, and Lemma~\ref{lem:adjoin}, applied to this second family, gives $H^{\le N}(J_2)=0$. If $b>N$, this already holds for $J_1$, and we take $J_2=J_1$.

Finally, Lemma~\ref{lem:completion} constructs $J$ containing $J_2+A^{>N}$. All the primitives $u_{s,i}$ and $v_{s,i}$ belong to the original algebra $A$. Neither the source nor its multiplication nor its differential has been replaced during the construction.
\end{proof}

\begin{remark}\label{rem:no-third}
The first stage has a normalization in $B$; the second does not need one, since its products lie above $N$. The relevant degree bounds are summarized below.
\begin{center}
\small
\begin{tabular}{@{}>{\raggedright\arraybackslash}p{0.56\linewidth}c@{}}
\toprule
Product & First possible degree\\
\midrule
First primitive times a positive-degree element & $(a-1)+q=a+q-1$\\
First primitive times a possibly nonclosed coefficient & $(a-1)+(2q-1)=a+2q-2$\\
Second primitive times a positive-degree element & $(a+q-2)+q=a+2q-2$\\
\bottomrule
\end{tabular}
\end{center}
The last two thresholds coincide, which gives the bound $N\le a+2q-3$. At degree $a+2q-2$, these products need not be inessential, so a third stage would require additional control.
\end{remark}

\section{Equality of the two nilpotency invariants and consequences}
\label{sec:main-proof}

\begin{theorem}[Refined form]\label{thm:refined}
Let $M=(\Lambda V,d)$ be minimal of finite type, with $V^{<q}=0$, $q\ge2$, $N\ge0$, and $H^{>N}(M)=0$. For $n\ge1$, if
\[
 \nilh(M)\le n,\qquad N\le a_n(M)+2q-3,
\]
then there is an acyclic differential ideal $J$ containing $\mm_M^{n+1}+M^{>N}$. In particular, $\Hnil(M)\le n$. Each of the conditions
\begin{equation}\label{eq:refined-variants}
 N\le\nu_{n+1}(M)+2q-3,
 \qquad N\le q(n+3)-3
\end{equation}
is sufficient.
\end{theorem}
\begin{proof}
If $H^+(M)=0$, the ideal $M^+$ suffices. We may therefore assume $N\ge1$. Lemma~\ref{lem:direct} gives a quasi-isomorphism $f:M\to B$ to an $n$-short target. It annihilates $I=\mm_M^{n+1}$. If $a_n=+\infty$, then $I$ is acyclic, and Lemma~\ref{lem:completion} still allows us to enlarge it to include $M^{>N}$.

Otherwise, take $a=a_n$. By~\eqref{eq:a-bound}, $a\ge q(n+1)\ge2q$. Minimality gives $d(M^{\le2q-2})=0$ by Lemma~\ref{lem:closed}. All the hypotheses of Theorem~\ref{thm:two} are therefore satisfied in $M$, and its conclusion provides $J$. The two variants follow from~\eqref{eq:a-bound}.
\end{proof}

\begin{proof}[Proof of Theorem~\ref{thm:main}]
Apply Theorem~\ref{thm:refined} at level $n$. The second condition in~\eqref{eq:refined-variants} is then the condition in~\eqref{eq:main}. The quotient by the constructed ideal is a surjective quasi-isomorphism by Lemma~\ref{lem:quotient}, so $\Hnil(M)\le n$. Comparison~\eqref{eq:nil-comparison} gives $n\le\Hnil(M)$, hence $\Hnil(M)=n$. The quotient vanishes in degrees greater than $N$ and is finite-dimensional because $M$ is of finite type.
\end{proof}

\begin{corollary}[Formulation for rational spaces]\label{cor:spaces}
Let $X$ be a $(q-1)$-connected rational space of finite type, where $q\ge2$, and suppose that $H^{>N}(X;\Q)=0$. Assume that $\nilh(X)=n$ for an integer $n\ge1$. If $N\le q(n+3)-3$, then
\[
 \Hnil(X)=n.
\]
\end{corollary}
\begin{proof}
The minimal model of $X$ is of finite type and satisfies $V^{<q}=0$; its cohomology is that of $X$ \cite{FHT}. Theorem~\ref{thm:main} applies.
\end{proof}

\Needspace{10\baselineskip}
At level three, the uniform values are as follows.
\begin{center}
\begin{tabular}{ccc}
\toprule
Connectivity at least & $q$ & Equality if $\nilh=3$ and\\
\midrule
$1$ & $2$ & $H^{>9}(M)=0$\\
$2$ & $3$ & $H^{>15}(M)=0$\\
$3$ & $4$ & $H^{>21}(M)=0$\\
$4$ & $5$ & $H^{>27}(M)=0$\\
\bottomrule
\end{tabular}
\end{center}
This table states sufficient conditions; it does not, by itself, locate the first possible separation. Compared with the uniform range $q(n+2)-2$ in Corollary~\ref{cor:toomer}, the gain is $q-1$ degrees.

\begin{corollary}[Application of an independent bound]\label{cor:transfer}
Let $m\ge0$ be an integer, and let $M$ be minimal and satisfy the degree and cohomology assumptions of Theorem~\ref{thm:main}. If an independent argument gives $\nilh(M)\le m+1$, then
\[
 N\le q(m+4)-3\quad\Longrightarrow\quad\Hnil(M)\le m+1.
\]
\end{corollary}
\begin{proof}
Apply the uniform upper-bound form of Theorem~\ref{thm:refined} with $n=m+1$.
\end{proof}
This is an application rule: a bound on $\nilh$ must first be established independently for the space or construction under study. The corollary then converts that bound into a bound on $\Hnil$ in the stated dimension range; it supplies no independent estimate for fibrations or cell attachments.

\begin{example}[Formality without a dimension restriction]
Suppose that $M$ is formal and its cohomology has finite cup-length $r$. There is a quasi-isomorphism $f:M\to(H(M),0)$ by Lemma~\ref{lem:direct}, applied to the formal type. It is surjective: every element of its target is closed, and surjectivity of $H(f)$ means here that it is the actual image of a cocycle of $M$. Thus $\Hnil(M)\le r$;~\eqref{eq:toomer-chain} gives $\nilh(M)=\Hnil(M)=r$. The dimension theorem is therefore not a necessary condition for equality.
\end{example}

\section{Why the hypothesis on homotopical nilpotency is necessary}
\label{sec:toomer-examples}

The first range required only $e_0\le n$. The second uses the normalization~\eqref{eq:normalized} in a short witness. We show that this difference is not merely a feature of the proof.

Fix $n\ge2$. Consider the minimal algebra
\begin{equation}\label{eq:F-toomer}
 F_n=\bigl(\Lambda(x_2,y_2,p_0,\ldots,p_{n+1}),d\bigr),
 \qquad |p_i|=2n+1,
\end{equation}
with $dx=dy=0$ and
\[
 dp_i=x^{n+1-i}y^i\qquad(0\le i\le n+1).
\]
Its differential squares to zero in the free algebra. Set $T=2n+3$ and
\[
 C_n=\Span\{xp_0,yp_0,yp_1,\ldots,yp_{n+1}\}\subseteq F_n^T,
 \qquad R_n=F_n/(C_n+F_n^{>T}).
\]
The denominator is a differential ideal: its positive-degree products and the differential of $C_n$ lie above $T$.

\begin{lemma}\label{lem:toomer-H}
The projection $F_n\to R_n$ induces a cohomology isomorphism through degree $T$, and $H^{T+1}(F_n)=0$. The cohomology of $R_n$ is given by
\[
 H^{2j}(R_n)=\Sym^j\langle[x],[y]\rangle\quad(0\le j\le n),
 \qquad H^T(R_n)\cong\Q^{n+1},
\]
with all other groups zero. A basis of the last group is represented by
\begin{equation}\label{eq:syzygies}
 \sigma_i=xp_i-yp_{i-1},\qquad 1\le i\le n+1.
\end{equation}
\end{lemma}
\begin{proof}
Below degree $2n+1$, only monomials in $x,y$ occur. The differential sends the $n+2$ elements $p_i$ bijectively onto the monomial basis of $\Sym^{n+1}\langle x,y\rangle$; hence degrees $2n+1$ and $2n+2$ contribute no cohomology.

In degree $T$, the source has a basis consisting of the $2n+4$ monomials $xp_i,yp_i$. The differential maps onto all of $\Sym^{n+2}\langle x,y\rangle$, which has dimension $n+3$. The $n+1$ cycles in~\eqref{eq:syzygies} are independent and span its kernel. The restriction of $d$ to $C_n$ is an isomorphism onto the image, since the images of the displayed basis are
\[
 x^{n+2},\ x^{n+1}y,\ x^ny^2,\ldots,y^{n+2}.
\]
Thus $C_n$ is a complement to the cycles in degree $T$. The projection sends the polynomial classes to their namesakes and sends $[\sigma_i]$ to $[xp_i]$ in $H(R_n)$, since $yp_{i-1}\in C_n$. These images form a cohomology basis through degree $T$, so the induced morphisms are isomorphisms. In degree $T+1$, all polynomial monomials are boundaries. Products of two $p_i$ begin in degree $4n+2>T+1$. This proves all the assertions.
\end{proof}

Lemma~\ref{lem:prefix}, proved in the appendix, gives a minimal model $E_n\qis R_n$ extending $F_n$, with no new generator in degrees at most $T$.

\begin{theorem}\label{thm:toomer-example}
For every $n\ge2$, we have
\begin{equation}\label{eq:toomer-counter}
 \cdim E_n=2n+3,
 \qquad e_0(E_n)=n<\nilh(E_n)=\Hnil(E_n)=n+1.
\end{equation}
\end{theorem}
\begin{proof}
The cohomological dimension follows from Lemma~\ref{lem:toomer-H}. The nonzero class $[x^n]\ne0$ gives $e_0(E_n)\ge n$. In degrees at most $T$, the only words of length at least $n+1$ are the polynomial monomials of degree $2n+2$, all of which are exact. A word of this length containing a $p_i$ would have degree at least $2n+1+2n>T$. The new minimal generators also lie above $T$. Since $H^{>T}(E_n)=0$, every cycle of $(E_n^+)^{n+1}$ is exact in $E_n$. Proposition~\ref{prop:toomer} therefore gives $e_0(E_n)=n$.

Suppose that $\nilh(E_n)\le n$ and choose a quasi-isomorphism $f:E_n\to B$ with $(B^+)^{n+1}=0$. The images $f(dp_i)$ are zero, so the $f(p_i)$ are closed. The group $H^{2n+1}(B)$ is zero; hence these images are exact. Since $f(x)$ and $f(y)$ are closed, all the $f(\sigma_i)$ are exact. This contradicts the injectivity of $H(f)$ on the nonzero classes in~\eqref{eq:syzygies}. Thus $\nilh(E_n)\ge n+1$.

Finally, $(E_n^+)^{n+2}$ starts in degree $2n+4>T$. Corollary~\ref{cor:degree} constructs an acyclic ideal containing this power, giving $\Hnil(E_n)\le n+1$. The inequality $\nilh\le\Hnil$ completes the proof.
\end{proof}

For $n=3$, this example has $H^{>9}=0$ and $e_0=3$, but $\Hnil=4$. It does not contradict Theorem~\ref{thm:main}, since its homotopical nilpotency is $4$. It shows that the upper-bound hypothesis $\nilh\le3$ in Theorem~\ref{thm:refined} cannot be replaced by $e_0\le3$.

\section{Optimality of the refined constant}
\label{sec:sharp}

\subsection{A family of three-short CDGAs}

We use the algebraic pattern of \cite[Chapters~11--12]{Parent2026}, with a new grading and a degree truncation adapted to the refined bound. To detect the refined endpoint, we must rule out every correction in the two-parameter family of critical primitives~\eqref{eq:all-primitives}.

Fix an odd integer $q\ge3$. Consider the free \emph{graded} algebra
\[
 \mathcal F_q=\Lambda(s_q,x_q,y_q,a_{2q-1},t_{4q-2},r_{5q-3}),
\]
and the derivation $\delta$ of degree $+1$ defined by
\begin{equation}\label{eq:separator-d}
 \delta s=\delta x=\delta y=0,\qquad
 \delta a=xy,\quad\delta t=ays,\quad\delta r=xt.
\end{equation}
The elements $s,x,y,a$ are odd, whereas $t,r$ are even. We have
\begin{equation}\label{eq:curvature}
 \delta^2t=xyys=0,\qquad \delta^2r=-xays\ne0
 \quad\text{in }\mathcal F_q.
\end{equation}
Thus $(\mathcal F_q,\delta)$ is not a CDGA. Set $N_q=7q-3$ and
\begin{equation}\label{eq:Dq}
 D_q=\left(\mathcal F_q\big/
       \bigl((\mathcal F_q^+)^4+\mathcal F_q^{>N_q}\bigr),\bar\delta\right).
\end{equation}
Both ideals in the denominator are stable under $\delta$. The square of an odd derivation is a derivation; its only nonzero value on the generators in~\eqref{eq:curvature} has length four. It therefore vanishes in the quotient. Thus $D_q$ is a simply connected, finite-type, three-short CDGA.

The following table gives a complete basis. The displayed order is chosen to simplify the differential formulas; passing to a canonical order introduces only Koszul signs.
\begin{center}
\begin{tabular}{cl}
\toprule
Degree & Basis of $D_q$\\
\midrule
$0$ & $1$\\
$q$ & $s,x,y$\\
$2q-1$ & $a$\\
$2q$ & $sx,sy,xy$\\
$3q-1$ & $sa,xa,ya$\\
$3q$ & $sxy$\\
$4q-2$ & $t$\\
$4q-1$ & $sxa,sya,xya$\\
$5q-3$ & $r$\\
$5q-2$ & $st,xt,yt$\\
$6q-3$ & $at,sr,xr,yr$\\
$6q-2$ & $sxt,syt,xyt$\\
$7q-4$ & $ar$\\
$7q-3$ & $ast,axt,ayt,sxr,syr,xyr$\\
\bottomrule
\end{tabular}
\end{center}
To check completeness, observe that each monomial has at most three letters, odd letters cannot be repeated, and $t^2$, $tr$, $r^2$ already lie above $N_q$. If exactly one $t$ or $r$ occurs, the remaining letters are chosen from the four odd ones; the degree bound gives precisely the table. Monomials containing neither $t$ nor $r$ correspond to subsets of at most three odd letters. The total dimension is $34$.

Besides the three differentials in~\eqref{eq:separator-d}, the only nonzero differentials in the table are
\begin{align}
 d(sa)&=-sxy,&d(at)&=xyt,\label{eq:table-d1}\\
 d(sr)&=-sxt,&d(yr)&=xyt,&d(ar)&=xyr-axt.\label{eq:table-d2}
\end{align}
In particular, $d(st)=d(yt)=0$; these identities follow from $s^2=y^2=0$, not merely from truncation.

\begin{lemma}\label{lem:D-cohom}
We have $\dim H(D_q)=20$, $\cdim D_q=7q-3$, and
\begin{equation}\label{eq:essential}
 [syt]\ne0\in H^{6q-2}(D_q),\qquad
 [axt],[ayt]\text{ are independent in }H^{7q-3}(D_q).
\end{equation}
More precisely, the series recording the cohomology dimensions is
\begin{align}\label{eq:betti}
 P_{D_q}(z)={}&1+3z^q+2z^{2q}+2z^{3q-1}+2z^{4q-1}\notag\\
             &+2z^{5q-2}+2z^{6q-3}+z^{6q-2}+5z^{7q-3}.
\end{align}
\end{lemma}
\begin{proof}
For $dv=w$, write $K(v,w)$ for the two-term complex $\Q v\xrightarrow{d}\Q w$. The differential formulas split $D_q$, as a complex of vector spaces, into the seven acyclic summands
\[
\begin{gathered}
 K(a,xy),\quad K(sa,-sxy),\quad K(t,ays),\quad K(r,xt),\\
 K(at,xyt),\quad K(sr,-sxt),\quad K(ar,xyr-axt),
\end{gathered}
\]
and a zero-differential complement with basis
\[
\begin{gathered}
1;\quad s,x,y;\quad sx,sy;\quad xa,ya;\quad sxa,xya;\\
st,yt;\quad xr,yr-at;\quad syt;\quad ast,axt,ayt,sxr,syr.
\end{gathered}
\]
To obtain this decomposition from the monomial basis, replace $yr$ by $yr-at$ and $xyr$ by $xyr-axt$, retaining $at$ and $axt$. Both changes are invertible. This splitting is a decomposition of complexes, not of algebras. The twenty displayed cocycles therefore give a basis of cohomology, with degrees recorded in~\eqref{eq:betti}.

In degree $6q-2$, the boundaries are spanned by $sxt,xyt$, so $syt$ represents a nonzero class. In the top degree, all six cochains in the table are closed, and the preceding degree is spanned by $ar$. The boundary subspace is the line $\Q(xyr-axt)$. A combination of $axt,ayt$ cannot belong to this line unless it is zero, by the independence of the monomials $xyr,axt,ayt$. This proves~\eqref{eq:essential}.
\end{proof}

\subsection{The genuinely minimal prefix}

Remove the generator $r$ and set
\[
 P_q=\bigl(\Lambda(s_q,x_q,y_q,a_{2q-1},t_{4q-2}),d\bigr),
 \qquad da=xy,\quad dt=ays.
\]
This time $d^2=0$ in the free algebra. The differentials are decomposable, and their factors have degrees strictly smaller than the degree of the generator being differentiated. Thus $P_q$ is minimal as an absolute algebra. Sending the generators to their namesakes defines a morphism $\phi_0:P_q\to D_q$.

\begin{lemma}\label{lem:rigid-prefix}
The morphism $H^j(\phi_0)$ is an isomorphism for $j\le5q-1$, and $H^{5q}(P_q)=0$. Hence there is a finite-type minimal model
\begin{equation}\label{eq:Mq}
 \phi:M_q=P_q\otimes\Lambda W\qis D_q,
 \qquad W^{\le5q-1}=0.
\end{equation}
In particular,
\begin{equation}\label{eq:critical-primitives}
 M_q^{5q-2}=\Span\{st,xt,yt\},\qquad
 Z^{5q-2}(M_q)=\Span\{st,yt\}.
\end{equation}
\end{lemma}
\begin{proof}
The first nonzero word of length four in $P_q$ is $sxya$, of degree $5q-1$. A word of length four containing $t$ has higher degree. The target generator $r$ has degree $5q-3$, and its first positive-degree product has degree $6q-3>5q-1$. Through degree $5q-1$, the only two changes to the complex are therefore the addition of $r$ and the removal of $sxya$.

The critical range is
\[
\begin{array}{c|ccc}
 &5q-3&5q-2&5q-1\\ \hline
 P_q&0&\Span\{st,xt,yt\}&\Q\,sxya\\
 D_q&\Q\,r&\Span\{st,xt,yt\}&0.
\end{array}
\]
In $P_q$, the element $d(xt)=-xays$ is a nonzero multiple of $sxya$, while $st,yt$ are closed. In $D_q$, the corresponding acyclic pair is $r\mapsto xt$. The two complexes have the same classes $[st],[yt]$ and no other cohomology in this range. The differential of an old cochain never has a component along $r$; adding $r$ therefore does not change the cycles in the preceding degree either. The common parts of the complexes are unchanged elsewhere in the range. The morphism $\phi_0$ preserves their cohomology representatives and sends $[st],[yt]$ to their namesakes. Thus the cohomology identifications just obtained are induced by $H^j(\phi_0)$.

In degree $5q$, there is no monomial of $P_q$ if $q>3$. For $q=3$, the only monomial is $at$, and $d(at)=xyt\ne0$. Thus $H^{5q}(P_q)=0$. All generators of $P_q$ have degree at most $5q-1$; Lemma~\ref{lem:prefix} applies with this integer. It gives~\eqref{eq:Mq}, followed by~\eqref{eq:critical-primitives}.
\end{proof}

\Needspace{8\baselineskip}\subsection{Excluding all possible primitives}

\begin{proposition}\label{prop:separator-lower}
No acyclic differential ideal of $M_q$ contains $\mm_{M_q}^4$.
\end{proposition}
\begin{proof}
Suppose that such an ideal $J$ exists. By acyclicity, the nonzero cocycle
\[
 z=d(xt)=-xays\in\mm_{M_q}^4\subseteq J
\]
has a primitive $u\in J^{5q-2}$. By~\eqref{eq:critical-primitives}, all possible primitives have the form
\begin{equation}\label{eq:all-primitives}
 u=xt+\alpha st+\beta yt,\qquad\alpha,\beta\in\Q.
\end{equation}
If $\alpha\ne0$, the element $yu\in J$ is closed: $dy=0$ and $yz=0$. We have
\[
 yu=yxt+\alpha yst=-xyt-\alpha syt,
 \qquad xyt=d(at).
\]
By~\eqref{eq:essential} and the quasi-isomorphism $\phi$, its class in $H(M_q)$ is $-\alpha[syt]\ne0$. An essential class cannot be represented by a cocycle of an acyclic ideal.

It remains to consider $\alpha=0$. The ideal then contains $au=axt+\beta ayt$. Each term is closed in the minimal algebra: its differential is a sum of terms containing $x^2$, $y^2$, or $a^2$, all of which vanish. Their images under $\phi$ have independent classes by~\eqref{eq:essential}. Thus $[au]\ne0$ for every $\beta$, giving the same contradiction.

These two cases cover every pair $(\alpha,\beta)$. No choice of primitive and no additional assumption on the ideals has been imposed.
\end{proof}

\begin{lemma}\label{lem:word-onset}
The initial degrees of the augmentation powers of $M_q$ satisfy
\[
 \nu_4(M_q)=a_3(M_q)=5q-1,
 \qquad \nu_5(M_q)=9q-3.
\]
\end{lemma}
\begin{proof}
In the prefix, the first word of length four is $sxya$, in degree $5q-1$. A word of this length involving a generator of $W$ has degree at least $5q+3q=8q$. Hence the component of $\mm^4$ in degree $5q-1$ is exactly $\Q sxya$. It is closed, and $\mm^4$ has no cochain in the preceding degree; it therefore defines a nonzero internal class. This gives $\nu_4=a_3=5q-1$.

A word of length five in the prefix must contain $t$, since the other four generators are odd. With a single $t$, the minimum is attained by $sxya\,t$, in degree $9q-3$. With at least two copies of $t$, the degree is higher. A word involving a generator of $W$ and four other factors has degree at least $5q+4q=9q$. Consequently $\nu_5=9q-3$.
\end{proof}

\begin{theorem}[Separation at the first degree outside the refined bound]\label{thm:sharp}
For every odd integer $q\ge3$, the minimal model $M_q$ satisfies
\begin{equation}\label{eq:separator-values}
 \cdim M_q=7q-3,
 \qquad e_0(M_q)=\nilh(M_q)=3<4=\Hnil(M_q),
\end{equation}
and $a_3(M_q)=5q-1$. Consequently, the constant $-3$ in the bound $N\le a_n(M)+2q-3$ cannot be replaced universally by $-2$.
\end{theorem}
\begin{proof}
The quasi-isomorphism $M_q\to D_q$ gives the cohomological dimension and $\nilh(M_q)\le3$. The essential cycle $axt\in\mm_{M_q}^3$ gives $e_0(M_q)\ge3$, since its class is killed in $M_q/\mm_{M_q}^3$. The chain~\eqref{eq:toomer-chain} then gives $e_0=\nilh=3$.

Proposition~\ref{prop:separator-lower} gives $\Hnil(M_q)>3$. Lemma~\ref{lem:word-onset} shows that $\mm_{M_q}^5$ starts in degree $9q-3>7q-3$. Corollary~\ref{cor:degree}, applied inside $M_q$, constructs an acyclic ideal containing this power. Hence $\Hnil(M_q)\le4$.

Finally,
\[
 7q-3=(5q-1)+2q-2=a_3(M_q)+2q-2.
\]
Equality of the two invariants therefore fails at the first integer beyond the refined bound throughout this family.
\end{proof}

\begin{corollary}[A uniform bound that fails]\label{cor:not-CP}
The assertion
\[
 V^{<q}=0,\quad \nilh(M)=n,\quad
 \cdim M\le q(2n+1)-3
 \quad\Longrightarrow\quad\Hnil(M)=n
\]
is false in general.
\end{corollary}
\begin{proof}
Take $n=3$ and $M=M_q$ in Theorem~\ref{thm:sharp}. The cohomological dimension is exactly $q(2\cdot3+1)-3$, whereas $\Hnil(M_q)=4$.
\end{proof}

\begin{remark}
For $q=3$, the three-short target has generators $s_3,x_3,y_3,a_5,t_{10},r_{12}$ and is truncated in degree $18$. Its minimal model satisfies $H^{18}\cong\Q^5$, $a_3=14$, and $\nilh=3<4=\Hnil$. The uniform bound gives $N\le15$; the refined bound, for these data, gives $N\le17$; separation occurs in degree $18$. These models do not satisfy Poincaré duality, since their top cohomology has dimension five.
\end{remark}

\section{Scope of the comparison}
\label{sec:conclusion}

Theorem~\ref{thm:main} gives a range in which an $n$-short model can be replaced by an $n$-short quotient of the minimal Sullivan algebra. Its proof constructs the required acyclic ideal inside that fixed algebra.

For the exact value $n=\nilh(M)$, the uniform equality range proved here is $N\le q(n+3)-3$. The proof provides more information when $a_n$ or $\nu_{n+1}$ is known. Only the constant in the refined bound is asserted to be optimal. The examples in Section~\ref{sec:sharp} do not determine, for every connectivity, the smallest cohomological dimension in which a three-short separation occurs.

The positive results apply to Poincaré duality types in the same range, since duality appears in none of the hypotheses. The counterexamples presented here do not belong to this subclass. Extending the range under duality assumptions would therefore require a new argument; such an extension is neither established nor refuted by the present work.

Finally, every upper bound on $\Hnil$ has been obtained from an acyclic ideal in the minimal model. The truncated targets in the examples are explicitly nonminimal; their purpose is to define a rational type and to bound its $\nilh$. The upper bound on $\Hnil$ is always proved separately. This distinction makes the dimension bound a strictification theorem rather than a mere change of notation.

\appendix
\section{Extending a minimal prefix from a prescribed degree}
\label{sec:appendix}

The following lemma starts the standard degreewise construction of \cite[Proposition~12.2 and its construction, pp.~144--146]{FHT} from a prescribed minimal prefix.

\begin{lemma}\label{lem:prefix}
Let $P=(\Lambda U,d)$ be a simply connected minimal Sullivan algebra of finite type, whose generators have degrees at most $r$. Let $g:P\to A$ be a morphism to a simply connected CDGA of finite type. Suppose that $H^j(g)$ is bijective for $j\le r$ and injective for $j=r+1$. There is a finite-type extension that is minimal as an absolute algebra,
\[
 g_\infty:P\otimes\Lambda W\qis A,
 \qquad W^{\le r}=0,
\]
extending $g$.
\end{lemma}
\begin{proof}
Proceed degree by degree, starting at $k=r+1$. At the beginning of step $k$, suppose that the source has been constructed with generators in degrees at most $k-1$, and that the comparison is bijective in cohomology through degree $k-1$ and injective in degree $k$.

Choose cocycles of $A^k$ whose classes form a basis of a complement to the cohomological image. Adjoin closed generators of degree $k$ and send them to these cocycles. The comparison becomes bijective in degree $k$. No positive-degree product of these generators occurs in degree $k+1$, since positive degrees start at two.

The kernel of the comparison in degree $k+1$ is represented by cocycles $z_i$ in the algebra preceding the new closed generators. Choose the $[z_i]$ as a basis of this kernel. Their images are exact in $A$; take $b_i\in A^k$ with $db_i=g(z_i)$. Adjoin generators $w_i$ of degree $k$ with $dw_i=z_i$, and send them to $b_i$.

Each $z_i$ is decomposable, since the source has no generator in degree $k+1$. In a decomposable monomial of degree $k+1$, all factors have degrees at most $k-1$, because positive degrees are at least two. The new differentials are therefore compatible with a Sullivan ordering and preserve absolute minimality.

A new closed combination of degree $k$ with a nonzero linear part in the $w_i$ would, after differentiation, give a relation among the independent classes $[z_i]$. Such a combination cannot exist. Bijectivity in degree $k$ is preserved, while the comparison becomes injective in degree $k+1$. This prepares the next step.

At the limit, every component and every cohomology group stabilize. The construction preserves finite type. The resulting extension is minimal, all its new generators have degree at least $r+1$, and its morphism to $A$ is a cohomology isomorphism in every degree.
\end{proof}

\begin{remark}
The lemma does not provide a surjection onto $A$. In the examples, it is used only to construct the full minimal model without changing the range containing the primitives under study. The upper bounds on $\Hnil$ come from ideals subsequently constructed in this minimal source.
\end{remark}

\section*{Declaration on the use of generative artificial intelligence}
During the preparation of this manuscript, the author used ChatGPT
(OpenAI) as an assistive tool for language editing, \LaTeX{} preparation,
and the verification and exploration of mathematical arguments. All
mathematical statements, proofs, references, and conclusions were
independently reviewed by the author, who assumes full responsibility for
the content of the article.
\end{document}